\documentclass{elsart}
\usepackage[english]{babel}
\usepackage{graphicx}
\usepackage{amsfonts}
\usepackage{amsmath}
\usepackage{color}

\usepackage{amsfonts}
\usepackage{eurosym}
\usepackage{hyperref}
\usepackage{amsbsy}
\usepackage{amssymb}
\usepackage{srcltx}
\usepackage{color}
\usepackage{eqnarray, amsmath}
\usepackage{graphicx}

\usepackage{mathrsfs}
\DeclareFontFamily{OT1}{pzc}{}
\DeclareFontShape{OT1}{pzc}{m}{it}{<-> s * [1.10] pzcmi7t}{}
\DeclareMathAlphabet{\mathpzc}{OT1}{pzc}{m}{it}

\usepackage{calrsfs}
\DeclareMathAlphabet{\pazocal}{OMS}{zplm}{m}{n}

\newtheorem{definition}{Definition}[section]
\newtheorem{theorem}{Theorem}[section]
\newtheorem{proposition}{Proposition}[section]

\let\join=\bowtie
\let\pa=\partial
\let\sb=\subset

\let\fa=\forall
\let\tim=\times
\let\mx=\mbox

\let\sm=\setminus
\let\ol=\overline

\let\ii=\infty

\let\rhu=\rightharpoonup

\def\gr{{\mathbb{R}}}

\def\bbM{{\mathbb{M}}}

\newcommand{\fk}[1]{\mathfrak{#1}}

\def\GM{{\mathbf{M}}}

\def\Gt{{\mathbf{t}}}

\def\cA{{\mathcal{A}}}
\def\cB{{\mathcal{B}}}

\def\cD{{\mathcal{D}}}
\def\cE{{\mathcal{E}}}
\def\cF{{\mathcal{F}}}
\def\cG{{\mathcal{G}}}
\def\cH{{\mathcal{H}}}

\def\cL{{\mathcal{L}}}

\def\cR{{\mathcal{R}}}

\def\cT{{\mathcal{T}}}

\let\a=\alpha
\let\b=\beta

\let\e=\epsilon

\let\m=\mu
\let\n=\nu
\let\o=\omega

\let\r=\rho
\let\s=\sigma
\let\t=\theta
\let\tt=\tau

\let\y=\eta

\let\vf=\varphi
\let\D=\Delta

\let\O=\Omega

\let\G=\Gamma

\newcommand{\set}[1]{\{#1\}}
\newcommand{\Set}[1]{\Bigl\{#1\Bigr\}}

\def\wid{\widetilde}

\def\div{\mathrm{div}}

\def\tr{\mathrm{tr}\,}

\newcommand{\pri}{\hbox to 10pt{\hfil\hbox to 0.4pt{%
			\vrule height5pt width0.4pt depth0pt}%
		\vrule width4pt height0.4pt depth0pt\hfil}}

\def\coxn#1#2{
	\ifcat#21%
	\ifcase#2\or\or
	#1^1,\,#1^2\or
	#1^1,\,#1^2,\,#1^3\or
	#1^1,\,#1^2,\,#1^3,\,#1^4
	\fi
	\else
	#1^1,\,#1^2,\dots,\,#1^#2
	\fi}

\def\xn#1#2{
	\ifcat#21%
	\ifcase#2\or\or
	#1_1,\,#1_2\or
	#1_1,\,#1_2,\,#1_3\or
	#1_1,\,#1_2,\,#1_3,\,#1_4
	\fi
	\else
	#1_1,\,#1_2,\dots,\,#1_#2
	\fi}

\def\cox2n#1#2#3{
	\ifcat#31%
	\ifcase#3\or\or
	#1_1#2_1,\,#1_2#2_2\or
	#1_1#2_1,\,#1_2#2_2,\,#1_3#2_3\or
	#1_1#2_1,\,#1_2#2_2,\,#1_3#2_3,\,#1_4#2_4
	\fi
	\else
	#1_1#2_1,\,#1_2#2_2,\dots,\,#1_#3#2_#3
	\fi}
\def\vx2n#1#2#3{(\cox2n{#1}{#2}{#3})}

\makeatletter \@addtoreset{equation}{section} \makeatother
\def\theequation{\arabic{section}.\arabic{equation}}

\def\pri{\hbox to 10pt{\hfil\hbox to 0.4pt{\vrule height5pt width0.4pt
			depth0pt}\vrule width5pt height0.4pt depth0pt\hfil}}

\input{tcilatex}

\begin{document}

\begin{frontmatter}%

\title{Crack occurrence in bodies with gradient polyconvex energies}%

\author{Martin Kru\v{z}\'{\i}k*, Paolo Maria Mariano**, Domenico Mucci***}%

\address{*Czech Academy of Sciences, Institute of Information Theory and Automation\\
	Pod Vod\'{a}renskou v\u{e}\v{z}\'{\i} 4,
	CZ-182 00 Prague 8,
	Czechia\\e-mail: kruzik@utia.cz}
\address{**DICEA, Universit\`{a} di Firenze\\via Santa Marta 3, I-50139 Firenze, Italy\\e-mail: paolomaria.mariano@unifi.it, paolo.mariano@unifi.it}%

\address{***DSMFI, Universit\`{a} di Parma\\ Parco Area delle Scienze 53/A, I-43134 Parma, Italy \\ e-mail:
domenico.mucci@unipr.it}%



\begin{abstract}
Energy minimality selects among possible configurations of a continuous body with and without cracks those compatible with assigned boundary conditions of Dirichlet-type. Crack paths are described in terms of curvature varifolds so that we consider both \textquotedblleft phase" (cracked or non-cracked) and crack orientation. The energy considered is gradient polyconvex: it accounts for relative variations of second-neighbor surfaces and pressure-confinement effects. We prove existence of minimizers for such an energy. They are pairs of deformations and varifolds. The former ones are taken to be $SBV$ maps satisfying an impenetrability condition. Their jump set is constrained to be in the varifold support.
\end{abstract}%

\begin{keyword}
Fracture, Varifolds, Ground States, Shells, Microstructures, Calculus of
Variations
\end{keyword}%

\end{frontmatter}%

\section{Introduction}
Deformation-induced material effects involving interactions beyond those of first-neighbor-type
can be accounted for by
considering, among the fields defining states,
higher-order deformation gradients.
In short, we can say that 
these effects emerge from \emph{latent} microstructures, intending those which do not strictly require to be represented by independent (observable) variables accounting for small-spatial-scale degrees of freedom. Rather they are such that \textquoteleft though its effects are felt in the balance equations, all relevant quantities can be expressed in terms of geometric quantities pertaining to apparent placements' \cite[p. 49]{C85}. A classical example is the one of Korteweg's fluid: the presence of menisci in capillarity phenomena implies curvature influence on the overall motion; it is (say) measured by second gradients \cite{K1} (see also \cite{DS} for pertinent generalizations). In solids length scale effects appear to be non-negligible for sufficiently small test specimens in various geometries and loading programs; in particular, when plasticity occurs in poly-crystalline materials, such effects are associated with grain size and accumulation of both randomly stored and geometrically necessary dislocations \cite{FMAH94}, \cite{FH}, \cite{Gud}.

These higher-order effects influence possible nucleation and growth of cracks because the corresponding hyperstresses enter the expression of Hamilton-Eshelby's configurational stress \cite{M17}. Here we refer to this kind of influence. We look at energy minimization and consider a variational description of crack nucleation in a body with second-gradient energy dependence.
We do not refer to higher order theories in abstract sense
(see \cite{DS} for a general setting, \cite{C85} for a physical explanations in terms of microstructural effects, \cite{M17} for a generalization of \cite{DS} to higher-order complex bodies), rather we consider a specific energy, in which we account for the gradient of surface variations and confinement effects due to the spatial variation of volumetric strain. Specifically, the energy we consider reads as 
\begin{equation}
\begin{aligned}
\cF(y,V;\mathcal{B}):=&\int_\cB \hat W\bigl( \nabla y(x),\nabla[\mathrm{cof}\nabla y(x)],\nabla[\det \nabla y(x)]\bigr)\,dx\\
&+\bar{a}\m_V(\mathcal{B})
+\int_{\cG_2(\cB)}a_{1}\Vert A\Vert^{\ol p}\,dV+a_{2}\Vert\pa V \Vert\,,
\end{aligned}
\end{equation}
with $\mathcal{B}$ a fit region in the three-dimensional real space, $\bar{a}$, $a_{1}$, and $a_{2}$ positive constants, $y:\mathcal{B}\longrightarrow \tilde{\mathbb{R}}^{3}$,  a special bounded variation map, a deformation that preserves the local orientation and is such that its jump set is contained in the support over $\mathcal{B}$ of a two-dimensional varifold $V$ with boundary $\partial V$ and generalized curvature tensor $A$. Such a support is a $2$-rectifiable subset of $\mathcal{B}$ with measure $\m_V(\mathcal{B})$. We identify such a set with a possible crack path, and the terms
$$\bar{a}\m_V(\mathcal{B})
+\int_{\cG_2(\cB)}a_{1}\Vert A\Vert^{\ol p}\,dV+a_{2}\Vert\pa V \Vert$$
represent a modification of the traditional Griffith energy \cite{Gri20}, which is just $\bar{a}\m_V(\mathcal{B})$ (i.e., it is proportional to the crack lateral surface area), so they have a configurational nature. The energy density $\hat{W}$ is assumed to be gradient polyconvex, according to the definition introduced in reference \cite{BKS}.

We presume that a minimality requirement for
$\cF(y,V;\mathcal{B})$ selects among cracked and free-of-crack configurations. To this aim we prove an existence theorem for such minima under Dirichlet-type boundary conditions; minimizing deformations satisfy also a condition allowing contact of distant body boundary pieces but avoiding self-penetration. This is the main result of this paper.

\section{Physical insight}

\subsection{Energy depending on $\nabla[\mathrm{cof}(\cdot)]$: a significant case}

The choice of allowing a dependence of the energy density $\hat{W}$ on $\nabla[\mathrm{cof}\nabla y]$
has physical ground: we consider an effect due to relative variations of neighboring surfaces.
Such a situation occurs, for example, in gradient
plasticity. We do not tackle directly its analysis here,
but in this section we explain just its geometric reasons.

In periodic and quasi-periodic crystals, plastic strain emerges from dislocation motion
through the lattice \cite{Phill01}, such phenomenon includes
meta-dislocations and their approximants in quasi-periodic lattices  \cite{FH11}, \cite{M19}. In polycristalline materials,
dislocations cluster at granular interstices obstructing or
favoring the re-organization of matter.
In amorphous materials other microstructural rearrangements determining plastic (irreversible) strain
occur. Examples are creation of voids, entanglement and disentanglement of polymers.

At macroscopic scale, the one of large wavelength approximation, a traditional way to account indirectly for the cooperative effects of irreversible microscopic mutations is to accept a multiplicative decomposition of the deformation gradient, commonly indicated by $F$, into so-called \textquotedblleft elastic", $F^{e}$, and \textquotedblleft plastic", $F^{p}$ factors \cite{Kr60}, \cite{L69}, namely $F=F^{e}F^{p}$, which we commonly name the \emph{ Kr\"{o}ner-Lee decomposition}. The plastic factor $F^{p}$ describes rearrangements of matter at a low scale, while $F^{e}$ accounts for macroscopic strain and rotation.

In such a view, the plastic factor $F^{p}$ indicates through its time-variation just how much (locally) the material goes far from thermodynamic equilibrium transiting from an energetic well to another, along a path in which the matter rearranges irreversibly.
In the presence of quasi-periodic atomic arrangements, as in quasicrystals, such a viewpoint requires extension to the phason field gradient \cite{LRT}, \cite{M06}.

Here, we restrict the view to cases in which just $F$ and its decomposition play a significant role: they include periodic crystals, polycrystals, even amorphous materials like cement or polymeric bodies, when we neglect at a first glance direct representation of the material microstructure in terms of appropriate morphological descriptors to be involved in Laundau-type descriptions coupled with strain.

With $\mathcal{B}$ a reference configuration for the body under scrutiny, at every its point $x$, the plastic factor $F^{p}$ maps the tangent space of $\mathcal{B}$ at $x$ into a linear space not otherwise specified, except assigning a metric $g_{\mathfrak{L}}$ to it---indicate such a space by $\mathfrak{L}_{F^{p}}$. Then, $F^{e}$ transforms such a space into the tangent space of the deformed configuration.

In general, the plastic factor $F^{p}$ allows us to describe an incompatible strain, so its curl does not vanish, i.e., $\mathrm{curl}F^{p}\neq 0$, unless we consider just a single crystal in which irrecoverable strain emerges from slips along crystalline planes. So, the condition
$\mathrm{curl}F^{p}\neq 0$, which may hold notwithstanding $\mathrm{curl}F= 0$, does not allow us to sew up one another the linear spaces $\mathfrak{L}_{F^{p}}$, varying $x\in\mathcal{B}$, so we cannot reconstruct an intermediate configuration, with the exception of a single crystal behaving as a deck of cards, parts of which can move along slip planes.
Of course, $\mathrm{curl}F^{p}= 0$ when $F^{e}$ reduces to the identity.

In modeling elastic-perfectly-plastic materials in large strain regime, we usually assume that the free energy density $\psi$ has a functional dependence on state variables of the type $\psi:=\tilde{\psi}(x,F,F^{p})$.
Further assumptions are listed below.
\begin{itemize}
	\item \emph{Plastic indifference}, which is invariance under changes in the reference shape, leaving unaltered the material structure (\emph{material isomorphims}); formally it reads as
	\begin{equation*}
	\tilde{\psi}(x,F,F^{p})=\tilde{\psi}(x,FG,F^{p}G),
	\end{equation*}
	for any orientation preserving unimodular second rank tensor $G$ mapping at every $x$ the tangent space $T_{x}\mathcal{B}$ of $\mathcal{B}$ at $x$ onto itself (the requirement $\mathrm{det}G=1$ ensures mass conservation along changes in reference configuration).
	\item \emph{Objectivity}: invariance with respect to the action of $SO(3)$ on the physical space; it formally reads
	\begin{equation*}
	\tilde{\psi}(x,F,F^{p})=\tilde{\psi}(x,QF,F^{p}),
	\end{equation*}
	for any $Q\in SO(3)$.
\end{itemize}

Plastic indifference implies
$\tilde{\psi}(x,F,F^{p})=\hat{\psi}(x,F^{e})$. Then, objectivity requires  $\hat{\psi}(x,F^{e})=\hat{\psi}(x,\tilde{C}^{e})$, with $\tilde{C}^{e}$ the right Cauchy-Green tensor $\tilde{C}^{e}=F^{e\mathrm{T}}F^{e}$, where $\tilde{C}^{e}=g_{\mathfrak{L}}^{-1}C^{e}$, with $C^{e}:=F^{e\ast}\tilde{g}F^{e}$ the pull-back in $\mathfrak{L}_{F^{p}}$ of the metric in $\mathcal{B}_{a}$ (the asterisk denotes formal adjoint, which coincides with the transpose when the metrics involved are flat). However, plastic indifference implies also $\tilde{\psi}(x,F,F^{p})=\hat{\psi}(x,F^{e}, \bar{g})$, where $ \bar{g}:=F^{p-\ast}gF^{p-1}$ is at each $x$ push-forward of the material metric $g$ onto the pertinent intermediate space $\mathfrak{L}_{F^{p}}$ through $F^{p}$. Since by the action of $G$ over the reference space $g$ becomes $G^{\ast}gG$, we get
$\bar{g}=F^{p-\ast}gF^{p-1}\overset{G}\longrightarrow(F^{p-\ast}G^{-\ast})G^{\ast}gG(G^{-1}F^{p-1})=\bar{g}$.

To account for second-neighborhood effects,
we commonly accept the free energy density to be
like $\hat{\psi}(x,F^{e}, D_{\alpha}F^{e})$ or
$\hat{\psi}(x,F^{e}, \bar{g}, D_{\alpha}F^{e})$, with $\alpha$ indicating that the derivative is computed with respect to coordinates over $\mathfrak{L}_{F^{p}(x)}$.

We claim here that
\emph{this choice---i.e., the presence of $D_{\alpha}F^{e}$ in the list of state variables---is related to the possibility of assigning energy to oriented area variations of neighboring staking faults when $\mathrm{det}F^{p}=1$}.

To prove the statement, first consider that the second-rank minors of $F^{p}$, collected in
$\mathrm{cof} F^{p}$, govern at each point $x$ the variations of oriented areas from the reference shape to the linear intermediate space associated with the same point. Neighboring staking faults
determine such variations in the microstructural arrangements collected in what we call plastic flows.
Since $\mathrm{det}F^{p}>0$, linear algebra tells us that $\mathrm{cof} F^{p}=(\mathrm{det}F^{p})F^{p-\ast}$, where $-\ast$ indicates adjoint of $F^{p-1}$.

Consequently, assigning energy to area variations due to first-neighbor staking faults, we may take a structure for the free energy as
\begin{equation*}
\psi:=\tilde{\psi}(x,F,F^{p},^{\urcorner}D\mathrm{cof} F^{p}),
\end{equation*}
where $D$ indicates the spatial derivative with respect to $x$, and the apex $^{\urcorner}$ indicates minor left adjoint operation of the first two indexes of a third order tensor (it corresponds to the minor left transposition when the metric is flat or the first two tensor components are both covariant or contravariant).
At least in the case
of volume-preserving crystal slips over planes ($\mathrm{det}F^{p}=1$), we have $\mathrm{cof} F^{p}=F^{p-\ast}$, whence we can write in operational form $D\mathrm{cof} F^{p}=F^{p-\ast}\otimes D$ so that $^{\urcorner}D\mathrm{cof} F^{p}=F^{p-1}\otimes D$. Under the action of $G$, describing a change in the reference shape, as above, we have
$F^{p-\ast}\otimes D\overset{G}\longrightarrow ((GF^{p-1})^{\ast}\otimes D)G$. Consequently, for volume-preserving plastic flows, the requirement of \emph{plastic invariance}
reads
\begin{equation*}
\tilde{\psi}(x,F,F^{p},F^{p-1}\otimes D)=\tilde{\psi}(x,FG,F^{p}G,((G^{-1}F^{p-1}))\otimes DG)
\end{equation*}
for any choice of $G$ with $\mathrm{det}G=1$.
The condition implies
\begin{equation*}
\begin{aligned}
\tilde{\psi}(x,F,F^{p},F^{p-\ast}\otimes D)&=\tilde{\psi}(x,FF^{p-1}, \bar{g},((FF^{p-1})\otimes D)F^{p-1})\\
&=\tilde{\psi}(x,FF^{p-1}, \bar{g},(DF^{e})F^{p-1})=\hat{\psi}(x,F^{e}, \bar{g}, D_{\alpha}F^{e}),
\end{aligned}
\end{equation*}
which concludes the proof.

Alternatively, if we choose
\begin{equation*}
\psi:=\tilde{\psi}(x,F,F^{p},D\mathrm{cof} F^{p}),
\end{equation*}
with the same argument as above we get
\begin{equation*}
\tilde{\psi}(x,F,F^{p},D\mathrm{cof} F^{p})=\hat{\psi}(x,F^{e}, \bar{g}, D_{\alpha}F^{e\ast}).
\end{equation*}

In our analysis here the density $\hat{W}$ is less intricate than $\tilde{\psi}(x,F,F^{p},D\mathrm{cof} F^{p})$, however, the analysis of its structure indicates
a fruitful path for dealing with more complex situations.

Finally, from now on we just assume flat metrics so that we write $\nabla$ instead of $D$, which appears to indicate the weak derivative of special bounded variation functions, a measure indeed. Also, we refer just to $F$ and do not consider the plasticity setting depicted by the multiplicative decomposition. Despite this, our choice of considering the gradient of $\mathrm{cof}F$ among the entries of $\tilde{W}$ is intended as an indicator of relative surface variation effects. Also, as already mentioned, the dependence of $\tilde{W}$ on $\nabla\mathrm{det}F$ is a way of accounting for confinement effects due to non-homogeneous volume variations (see \cite{BMM} for a pertinent analysis in small strain regime).

Explanations a part are necessary for justifying the representation of cracks in terms of varifold, which are special vector-valued measures.

\subsection{Cracks in terms of varifolds}

Take a reference configuration $\mathcal{B}$ of a body that can be cracked, and a set
of its infinitely many copies differing one another just by a possible crack path, each a $\mathcal{H}^{2}$-rectifiable set. In this reference picture, each crack path can be considered fictitious, i.e., the projection over $\mathcal{B}$ of the real crack occurring in the deformed shape; in other words, it can be considered as a shadow over a wall. Assigned boundary conditions, a question can be whether a crack may occur so that the deformed configuration is in one-to-one correspondence with at least one of the infinitely many reference configurations just depicted.

We may imagine of giving an answer by taking an expression of the energy including both bulk and crack components, asking its minimality as a criterion of selecting among configurations with or without cracks.
This is what has been proposed in reference \cite{FM98} taking Griffith's energy \cite{Gri20} as the appropriate functional. This minimality criterion is also a first step
to approximate a cracking process \cite{FM98}. To this aim we may select a finite partition of the time interval presuming to go from the state at instant $k$ to the one at $k+1$ by minimizing the energy. In principle, the subsequent step should be computing the limit as partition interval goes to zero. This path rests on
De Giorgi's notion of minimizing movements \cite{DeG93}.

In the minimum problem, deformation and crack paths
are the unknowns. A non-trivial difficulty emerges: in three dimensions we cannot control minimizing sequences of surfaces. A way of overcoming the difficulty is to consider as unknown just the deformation taken, however, in the space of those special functions with bounded variations, which are orientation preserving. We give their formal definition in the next section. Here, we just need to know that they admit a jump set with non-zero $\mathcal{H}^{2}$ measure. Once found minima of such a type, we identify
the crack path with the deformation jump set \cite{DMT002}. Although such a view is source of
nontrivial analytical problems and pertinent results \cite{DMT002}, it does not cover cases in which portions of the crack margins are in contact but material bonds across them are broken. To account for these phenomena, we need to recover the original proposal
in reference \cite{FM98}, taking once again separately deformations and crack paths. However, the problem of controlling minimizing sequences of surfaces or more irregular crack paths reappears. A way of overcoming it is to select minimizing sequences with bounded curvature because this restriction would avoid surface blow up. This is the idea leading to the representation of cracks in terms of varifolds.

Take $x\in\mathcal{B}$, the question to be considered
is not only whether $x$ belongs to a potential crack path or not but also, in the affirmative case, what is the tangent (even in approximate sense) of the crack there, among all planes $\Pi$ crossing $x$.
Each pair $(x,\Pi)$ can be viewed as a typical point of a fiber bundle $\mathcal{G}_{k}(\mathcal{B})$, $k=1,2$,
with natural projector $\pi:\mathcal{G}_{k}(\mathcal{B})\longrightarrow\mathcal{B}$ and typical fiber $\pi^{-1}(x)=\mathcal{G}_{k,3}$ the Grassmanian of 2D-planes or straight lines associated with $\mathcal{B}$. A $k$-varifold over $\mathcal{B}$ is a non-negative Radon measure $V$ over the bundle $\mathcal{G}_{k}(\mathcal{B})$ \cite{Alm65}, \cite{All72}, \cite{All75}, \cite{Mant96}.
For the sake of simplicity, here we consider just $\mathcal{G}_{2}(\mathcal{B})$, avoiding one-dimensional crack in a $3D$-body. The generalization to include $1D$ cracks is straightforward. Itself, $V$ has a projection $\pi_\#V$ over $\mathcal{B}$, which is a Radon measure over
$\mathcal{B}$, indicated for short by $\mu_{V}$.
Specifically, we may consider varifolds supported by 
$\mathcal{H}^{2}$-rectifiable subsets of $\mathcal{B}$,
i.e., by potential crack paths. We look at those varifolds admitting a certain notion of generalized curvature (its formal definition is in the next section) and parametrize through them the set of infinitely many reference configurations described above. Rather than sequences of cracks, we consider sequences of varifolds. The choice allows us to avoid the problem of controlling sequences of surfaces but forces us to include the varifold and its curvature in the energy, leading (at least in the simplest case) to a variant of Griffith's energy augmented by
$$\int_{\cG_2(\cB)}a_{1}\Vert A\Vert^{\ol p}\,dV+a_{2}\Vert\pa V \Vert$$
with respect to the traditional term just proportional to the surface crack area, namely $\bar{a}\m_V(\mathcal{B})$. Such a view point has been
introduced first in references \cite{GMMM10} and \cite{M10} (see also \cite{GMM10}).

The discussion in this section justifies a choice of
a energy functional like $\mathcal{F}(y,V;\mathcal{B})$, indicated above, which we analyze in the next sections.

\section{Background analytical material}
\subsection{Some notation}
For $G:\gr^n\to\gr^N$ a linear map, where $n\geq 2$ and $N\geq 1$, we indicate also by
$ G=(G^j_i)$, ${j=1,\ldots,N}$, $i=1,\ldots n $,
the $(N\tim n)$-matrix representing $G$ once we have assigned bases $(e_1,\ldots,e_n)$ and $(\e_1,\ldots,\e_N)$ in $\gr^n$ and $\gr^N$, respectively.

For any ordered multi-indices $\a$ in $\{1,\ldots,n\}$ and
$\b$ in $\{1,\ldots,N\}$ with length $|\a|=n-k$ and $|\b|=k$, we denote by $G^{\b}_{\ol\a}$
the $(k\tim k)$-submatrix of $G$ with rows
$\b=(\b_1,\ldots,\b_k)$ and columns
$\ol\a=(\ol\a_1,\ldots,\ol\a_k)$, where $\ol\a$ is the element which complements
$\a$ in $\{1,\ldots,n\}$, and $0\leq k\leq\ol n:=\min\{n,N\}$.
We also denote by
$$M^\b_{\ol\a}(G):= \det G^{\b}_{\ol\a}$$
the determinant of $G^{\b}_{\ol\a}$\,, and set $M^0_{0}(G):=1$. Also, the Jacobian $|M(G)|$ of the graph map $x\mapsto (Id\join G)(x):=(x,G(x))$  from $\gr^n$ into $\gr^n\tim\gr^N$ satisfies
\begin{equation}\label{MG}
|M(G)|^2:=\sum_{|\a|+|\b|=n}M^\b_{\ol\a}(G)^2. \end{equation}
\subsection{Currents carried by approximately differentiable maps}
Let $\O\sb\gr^n$ be a bounded domain, with $\cL^n$ the pertinent Lebesgue measure.
For $u:\O\to \gr^N$ an $\cL^n$-a.e. approximately differentiable map, we denote by $\nabla u(x)\in\gr^{N\tim n}$ its
approximate gradient at a.e. $x\in\O$. The map $u$ has a \emph{Lusin
representative} on the subset $\wid \O$ of Lebesgue points pertaining to both
$u$ and $\nabla u$. Also, we have $\cL^n(\O\sm\wid\O)=0$.

In this setting, we write $u\in\cA^1(\O,\gr^N)$ if
\begin{itemize}
	\item $\nabla u\in L^1(\O,\bbM^{3\times 2})$ and
	
	\item $M^\b_{\ol\a}(\nabla u)\in L^1(\O)$ for any ordered multi-indices $\a$ and
	$\b$ with $|\a|+|\b|=n$.
\end{itemize}
\par The {\em graph} $\cG_ u$ of a map $u\in\cA^1(\O,\gr^N) $ is defined by
\begin{equation*}
\cG_ u:=\Set{ (x,y) \in \O\times \gr^N \mid x\in \wid\O\,,\ y=\widetilde{u}(x)},
\end{equation*}
where $\wid{u}(x) $ is the Lebesgue value of $u$. It turns out that $\cG_u$ is a countably $n$-rectifiable set of $\O\tim\gr^N$, with $\cH^n(\cG_u)<\infty$. The
approximate tangent $n$-plane at $(x,\wid u(x))$ is generated by
the
vectors $\Gt_i(x)=(e_i,\pa_i u(x))\in\gr^{n+N}$, for $i=1,\ldots,n$, where the partial derivatives are the column vectors of the gradient matrix $\nabla u$, and we take
$\nabla u(x)$ as the Lebesgue value of $\nabla u$ at $x\in\wid\O$.

The unit $n$-vector
$$ \xi(x):=\frac{\Gt_1(x)\wedge \Gt_2(x)\wedge\cdots\wedge \Gt_n(x)}{|\Gt_1(x)\wedge \Gt_2(x)\wedge\cdots\wedge \Gt_n(x)|} $$
provides an orientation to the graph $\cG_u$.
\par For $\cD^k(\O\tim\gr^N)$ \emph{the vector space of compactly supported smooth $k$-forms in} $\O\tim\gr^N$,
and $\cH^k$ the $k$-dimensional Hausdorff measure, one defines the current
$G_u$ carried by the graph of $u$ through the integration of $n$-form on $\cG_u$, namely
$$ \langle G_u,\o \rangle:=\int_{\cG_u}\langle \o,\xi\rangle\,d\cH^n\,,\qquad \o\in \cD^n(\O\tim\gr^N),$$
where $\langle ,\rangle$ indicates the duality pairing.
Consequently, by definition $G_{u}$ is an element
of the (strong) dual of the space $\cD^n(\O\tim\gr^N)$.
Write $\cD_n(\O\tim\gr^N)$ for such a dual space.
Any element of it is called a \emph{current}.

By writing $U$ for a open set in $\mathbb{R}^{n+N}$,
we define \emph{mass} of $T\in\cD_k(U)$ the number
$$\GM(T):=\sup\{\langle T,\o\rangle\mid \o\in\cD^k(U)\,,\,\,\Vert\o\Vert\leq 1\} $$
and call a {\em boundary} of $T$ the $(k-1)$-current $\partial T$ defined by
$$\langle \pa T,\y\rangle := \langle T,d\y \rangle, \qquad \y\in \cD^{k-1}(U),$$
where $d\y$ is the differential of $\y$.

A {\em weak convergence} $T_h\rightharpoonup T$ in the sense of currents in $\cD_k(U)$ is defined through the formula
$$ \lim_{h\to\ii}\langle T_{h},\o \rangle = \langle T,\o \rangle\qquad\fa\,\o\in \cD^k(U)\,.$$
If $T_h\rightharpoonup T$, by lower semicontinuity we also have
$$\GM(T)\leq\liminf_{h\to\ii}\GM(T_h)\,.$$

With these notions in mind, we say that $G_u$ is an \emph{integer multiplicity} (in short i.m.) \emph{rectifiable current} in $\cR_n(\O\tim\gr^N)$, with finite mass
$\GM(G_u)$ equal to the area $\cH^n(\cG_u)$ of the $u$-graph.
According to \eqref{MG}, since the Jacobian $|M(\nabla u)|$ of the graph map $x\mapsto (Id\join u)(x)=(x,u(x))$ is equal to
$|\Gt_1(x)\wedge \Gt_2(x)\wedge\cdots\wedge \Gt_n(x)|$, by the area formula
$$ \langle G_u,\o \rangle=
\int_\O (Id\join u)^\#\o=\int_\O
\langle\o(x,u(x)),M(\nabla u(x))\rangle\, dx
$$
for any $\o\in \cD^n(\O\tim\gr^N)$, so that
$$ \GM(G_u)=\cH^n(\cG_u)=\int_\O|M(\nabla u)|\,dx < \infty\ . $$

\par If $u$ is of class $C^2$, the Stokes theorem implies
$$ \langle \pa G_u,\y\rangle=\langle G_u,d\y\rangle=\int_{\cG_u}d\y=\int_{\partial \cG_u}\y=0 $$
for every $\y\in \cD^{n-1}(\O\times \gr^N)$, i.e., the null-boundary condition
\begin{equation}\label{bdryzero} (\pa G_u)\pri\O\tim\gr^N=0\,. \end{equation}
Such a property \eqref{bdryzero} holds true also
for Sobolev maps $u\in W^{1,\ol n}(\O,\gr^N)$, by approximation. However, in
general, the boundary $\pa G_u$ does not vanish and may not have finite mass in $\O\tim\gr^N$.
On the other hand, if $\pa G_u$ has finite mass, the boundary rectifiability theorem states that $\pa G_u$ is an i.m. rectifiable current in $\cR_{n-1}(\O\tim\gr^N)$.
An extended treatment of currents is in the two-volume treatise \cite{GMS98}.
\subsection{Weak convergence of minors}
Let $\{u_h\}$ be a
sequence in $\cA^1(\O,\gr^N)$.
\par Take $N=1$, i.e., consider real-valued maps $u$. Suppose also to have in hands sequences $\{u_h\}$ and $\{\nabla u_h\}$ such that $u_h\to u$ strongly in
$L^1(\O)$ and $\nabla u_h\rightharpoonup v$ weakly in
$L^1(\O,\gr^n)$, where $u\in L^1(\O)$ is an
a.e. approximately differentiable map and $v\in
L^1(\O,\gr^n)$. In general, \emph{we cannot conclude} that $v=\nabla
u$ a.e. in $\O$. The question has a positive answer provided that $\{u_h\}$ is a sequence in $W^{1,1}(\O)$.
Notice that, when $N=1$, the membership of a function $u\in\cA^1(\O,\gr)$ to the Sobolev space $W^{1,1}(\O)$ is equivalent to the null-boundary condition
\eqref{bdryzero}.
\par When $N\geq 2$, assume that $u_h\to u$
strongly in $L^1(\O,\gr^N)$, with $u$ some a.e. approximately differentiable $ L^1(\O,\gr^N)$ map. Presume also that $M^\b_{\overline\a}(\nabla u_h)\rightharpoonup
v^\b_{\ol\a}$  weakly in $L^1(\O)$, with $v^\b_{\ol\a}\in L^1(\O)$, for every multi-indices $\a$ and
$\b$, with $\vert\a\vert+\vert\b\vert=n$.
A \emph{sufficient condition} ensuring that
$v^\b_{\ol\a}=M^\b_{\ol\a}(\nabla u)$ a.e. is again the validity of equation \eqref{bdryzero} for each $u_h$.
\par We can weaken such a condition by requiring a mass
control on $G_{u_h}$ boundaries  of the type
\begin{equation}\label{MbdGuk} \sup_h\GM((\pa G_{u_h})\pri\O\tim\gr^N)<\ii\,, \end{equation}
as stated by Federer-Fleming's closure theorem \cite{FF}, which refers to sequences of graphs $G_{u_h}$ which have
equi-bounded masses, $\sup_h\GM(G_{u_h})<\ii$ and satisfy the condition \eqref{MbdGuk}
\cite[Vol.~I, Sec.~3.3.2]{GMS98}.
\begin{theorem}\label{TclosG}{\bf (Closure theorem).}  Let
	$\{u_h\}$ be a sequence in $\cA^1(\O,\gr^N)$ such that
	$u_h\to u$ strongly in $L^1(\O,\gr^N)$ to an a.e. approximately differentiable map $u\in
	L^1(\O,\gr^N)$. For any multi-indices $\a$ and $\b$ with $\vert\a\vert+\vert\b\vert=n$, assume
	$$M^\b_{\ol\a}(\nabla u_h)\rightharpoonup v^\b_{\ol\a}\qquad{\text{weakly in }}L^1(\O), $$
	with $v^\b_{\ol\a}\in L^1(\O)$.
	If the bound $\eqref{MbdGuk}$ holds, the inclusion $u\in\cA^1(\O,\gr^N)$ holds
	and, for every $\a$ and $\b$,
	\begin{equation}\label{vabe} v^\b_{\ol\a}(x)=M^\b_{\ol\a}(\nabla
	u(x))\qquad \cL^n{\mx{-a.e in }}\O\,. \end{equation}
	Moreover, we find $G_{u_h}\rightharpoonup G_u$ weakly in $\cD_n(\O\tim\gr^N)$, and also
	$$\begin{array}{rl}\GM(G_u) \leq &\liminf\limits_{h\to\ii}\GM(G_{u_h})<\ii \\
	\GM((\pa G_{u})\pri\O\tim\gr^N)\leq &\liminf\limits_{h\to \ii}\GM((\pa G_{u_h})\pri\O\tim\gr^N)<\ii\,. \end{array}$$
\end{theorem}
\subsection{Special functions of bounded variation}
A summable function $u\in L^1(\O)$ is said to be \emph{of bounded variation} if the distributional derivative $Du$ is a finite measure in $\O$.
Such a function $u$ is approximately differentiable $\cL^n$-a.e. in $\O$. Its approximate gradient $\nabla u$ agrees with the
Radon-Nikodym derivative density of $Du$ with respect to $\cL^n$.
Then, the decomposition $Du=\nabla u\,\cL^n+D^su$ holds true, where the component $D^su$ is singular with respect to $\cL^n$.
Also, the \emph{jump set} $S(u)$ of $u$ is a countably $(n-1)$-rectifiable subset of $\O$ that agrees $\cH^{n-1}$-essentially with the complement of $u$ Lebesgue's set.
If, in addition, the singular component $D^su$ is concentrated on the jump set $S(u)$, we say that $u$ is a {\em special function of bounded variation}, and write in short $u\in SBV(\O)$.
\par A vector valued function $u:\O\to\gr^N$ belongs to the class $SBV(\O,\gr^N)$ if all its components $u^j$ are in $SBV(\O)$.
In this case, $Du=\nabla u\,\cL^n+D^su$, where the approximate gradient $\nabla u$ belongs to $L^1(\O,\gr^{N\tim n} )$, and the jump set $S(u)$ is defined component-wise as 
in the scalar case, so that
$D^su=(u^+-u^-)\otimes\n\cH^{n-1}\pri S(u)$, where $\n$ is an unit normal to $S(u)$ and $u^\pm$ are the one-sided limits at $x\in S(u)$. Therefore, for each Borel set $B\subset\Omega$ we get
$$ |Du|(B)=\int_B|\nabla u|\,dx+\int_{ B\cap S(u)} |u^+ -u^- | \, d\cH^{n-1}\,.  $$
\par
Compactness and lower semicontinuity
results hold in $SBV$. The treatise \cite{AFP00} offers an accurate analysis of $SBV$ landscape.
Here, we just recall that the
compactness theorem in \cite{A2} relies on a
generalization of the following characterization of
$SBV$ functions with $\cH^{n-1}$-rectifiable jump sets. 

According to reference 
\cite{ABG}, we denote by ${\cT}(\O\tim\gr)$ the class of
$C^1$-functions \,$\vf(x,y)$\, such that \,$\vert\vf\vert+\vert
D\vf\vert$\, is bounded and the support of $\vf$\, is contained in
$K\tim\gr$ for some compact set $K\sb\O$.
\begin{proposition}\label{PABG} Take $u\in BV(\O)$. Then, $u\in SBV(\O)$, with
	$\cH^{n-1}(S(u))<\ii$, if and only if for every $i=1,\ldots,n$
	there exists a Radon measure $\m_i$ in $\O\tim\gr$ such
	that
	$$
	\int_\O\Bigl(\frac{\pa\vf}{\pa
		x_i}(x,u(x))+ \frac{\pa\vf}{\pa
		y}(x,u(x))\,\pa_i{u}(x)\Bigr)\,dx=
	\int_{\O\tim\gr}\vf\,d\m_i $$
	for any \,$\vf\in{\cT}(\O\tim\gr)$. In this case, we have
	$$\m_i= -(Id\join u^+)_\#(\n_i\cH^{n-1}\pri
	S(u))+(Id\join u^-)_\#(\n_i\cH^{n-1}\pri S(u))\,. $$
\end{proposition}
\par As a consequence, we infer that if a sequence $\{u_h\}\in\cA^1(\O,\gr^N)$ satisfies
$$\sup_h\Bigl( \Vert u_h\Vert_\ii+\int_\O|M(\nabla u_h)|^p\,dx\Bigl)<\ii \,,\qquad p>1$$
and the boundary mass bound \eqref{MbdGuk}, the inclusion $\{u_h\}\in SBV(\O,\gr^N)$ and the $SBV$ compactness theorem hold. In fact, by
Proposition~\ref{PABG} we get
$$\cH^{n-1}\pri S(u_h)\leq \pi_\#|\pa G_{u_h}|(B)\qquad\fa\,h $$
where $\pi:\O\tim\gr^N\to\O$ is the projection onto the first $n$ coordinates, and $|\cdot|$ the total variation,
so that $\pi_\#|\pa G_u|(B)=|\pa G_u|(B\tim\gr^N)$ for each Borel set $B\sb\O$.
\subsection{Generalized functions of bounded variation}
When the bound $\sup_h\Vert u_h\Vert_\ii<\ii$ fails, the SBV compactness theorem cannot be applied.
This happens, e.g., if $u_h=\nabla y_h$ for some sequence $\{y_h\}\sb W^{1,p}(\O)$. When such sequences play a role in the problems analyzed, we find it convenient to call upon
{\em generalized special functions of bounded variation}, the class of which is commonly denoted by $GSBV$.

To define them, first write $SBV_{\textrm{loc}}(\O)$ for functions $v:\O\to\gr$ such that
$v_{\vert K}\in SBV(K)$ for every compact set $K\subset\O$.
\begin{definition} A function $u:\O\to\gr^N$ belongs to the class $GSBV(\O,\gr^N)$ if $\phi\circ u\in SBV_{\textrm{loc}}(\O)$ for every $\phi\in C^1(\gr^N)$ with the support of $\nabla\phi$ compact. \end{definition}
The following compactness theorem holds.
\begin{theorem}\label{TGSBV} Let $\{u_h\}\sb GSBV(\O,\gr^N)$ be such that
	$$ \sup_h \Bigl( \int_\O \bigl(|u_h|^p+|\nabla u_h|^p\bigr)\,dx +\cH^{n-1}(S_{u_h})\Bigr)<\ii$$
	for some real exponent $p>1$. Then, there exists a function $u\in GSBV(\O,\gr^N)$ and a (not relabeled) subsequence of $\{u_h\}$ such that $u_h\to u$ in $L^p(\O,\gr^N)$, $\nabla u_h\rightharpoonup \nabla u$ weakly in $L^p(\O,\gr^{N\tim n})$, and $\cH^{n-1}\pri S({u_h})$ weakly converges in $\O$ to a measure $\m$ greater than $\cH^{n-1}\pri S(u)$. \end{theorem}
\subsection{Curvature varifolds with boundary}
We now turn to the physical dimension $n=3$ and denote by $\cB$ a connected bounded domain in $\gr^3$ with surface-like boundary that can be oriented by the outward unit normal to within a finite number of corners and edges. In this setting, we take the deformation as a map $y:\mathcal{B}\longrightarrow\tilde{\mathbb{R}^{3}}$, where $\tilde{\mathbb{R}^{3}}$ is a isomorphic copy of $\mathbb{R}^{3}$, the isomorphism given by the identification. Such a distinction is necessary for example when we consider changes in observers (which are frames on the entire ambient space) leaving invariant the reference configuration, which is $\mathcal{B}$ in this case.

\begin{definition}
A general 2-varifold in $\cB$ is a non-negative Radon measure on
the trivial bundle $\cG_2(\cB):=\cB\times \cG_{2,3}$, where
$\cG_{2,3}$ is the
Grassmanian manifold of $2$-planes $\Pi$ through the origin in $\gr^{3}$.
\end{definition}

\par If $\fk{C}$ is a 2-rectifiable subset of $\cB$, for $\cH^2\pri\fk{C}$ a.e. $x\in \cB$
there exists the approximate tangent $2$-space $T_{x}\fk{C}$ to
$\fk{C}$ at $x$. We thus denote by $\Pi(x)$
the $3\tim 3$ matrix that identifies the orthogonal
projection of $\gr^3$ onto $T_{x}\fk{C}$ and define
\begin{equation}\label{rem}
V_{\fk{C},\t}(\vf) :=\int_{\cG_{2}(\cB)}\vf(x,\Pi) \, dV_{\fk{C},\t}(x,\Pi):=
\int_{\fk{C}}\t(x) \vf(x,\Pi(x)) \, d\cH^2(x)
\end{equation}
for any $\vf \in C_{c}^{0}(\cG_2(\cB))$, where
$\t\in L^1(\fk{C},\cH^2) $ is a nonnegative
density function.
If $\t$ is integer valued, then $V=V_{\fk{C},\theta }$ is said to be the
{\em integer rectifiable varifold} associated with
$(\fk{C},\t,\cH^2)$.
\par The {\em weight measure} of $V$
is the Radon measure in $\cB$ given by $\m_V:=\pi_\# V$, where $\pi:\cG_2(\cB)\to\cB$ is the canonical projection. Then, we have
$\m_V=\t\,\cH^2\pri\fk{C}$ and call
$$\Vert V\Vert:=V(\cG_2(\cB))=\m_V(\cB)=\int_{\fk{C}}\t\, d\cH^2\,$$
a \emph{mass} of $V$.
\begin{definition}\label{Dvarif} An integer rectifiable
	$2$-varifold $V=V_{\fk{C},\theta }$ is called a \emph
	{curvature $2$-varifold with boundary} if
	there exist a function
	$A\in L^{1}(\cG_2(\cB),\gr^{3*}\otimes\gr^3\otimes
	\gr^{3*})$, $A=(A^{\ell i}_j)$, and a $\gr^3$-valued measure $\pa V$ in $\cG_2(\cB)$ with finite mass $\Vert\pa V\Vert$,
	such that
	\begin{equation*}
	\int_{\cG_2(\cB)}(\Pi D_x\vf + A D_{\Pi} \vf + \vf \text{ }^{t}\tr(A I)) \, dV(x,\Pi)
	= - \int_{\cG_{2}(\cB)} \vf \, d\pa  V(x,\Pi)
	\end{equation*}%
	for every $\vf\in C_{c}^{\ii}(\cG_2(\cB))$.
	Moreover, for some real exponent $\ol p>1$, the subclass of curvature $2$-varifolds with boundary
	such that $|A|\in L^{\ol p}(\cG_2(\cB))$ is indicated by
	$CV^{\ol p}_2(\cB)$.
\end{definition}
\par Varifolds in $CV^{\ol p}_2(\cB)$ have generalized curvature in $L^{\ol p}$ \cite{Mant96}. Therefore,
Allard's compactness theorem applies (see \cite{All72}, \cite{All75}, but also \cite{Alm65}):
\begin{theorem}\label{Tcomp1}
	For $1<{\ol p}<\ii$, let $\set{V^{(h)}}\sb CV^{\ol p}_2(\cB)$ be a sequence of
	curvature $2$-varifolds $V^{(h)}=V_{\fk{C}_{h},\theta_{h}}$ with
	boundary. The corresponding curvatures and boundaries are
	indicated by $A^{(h)}$ and $\pa
	V^{(h)}$, respectively. Assume that there exists a real constant $c>0$ such that for
	every $h$
	\begin{equation*}
	\m_{V^{(h)}}(\cB) +
	\Vert\pa V^{(h)}\Vert + \int_{\cG_{2}(\cB)} |A^{(h)}|^{\ol p}\, dV^{(h)}\le c.
	\end{equation*}%
	Then, there exists a (not relabeled) subsequence of $\set{V^{(h)}}$
	and a $2$-varifold $V=V_{\fk{C},\theta}\in CV^{\ol p}_2(\cB)$, with
	curvature $A$ and boundary $\pa V$, such that
	\begin{equation*}
	V^{(h)} \rhu V,\quad A^{(h)}\,dV^{(h)}\rhu A\,dV,
	\qquad \pa V^{(h)} \rhu \pa V,
	\end{equation*}%
	in the sense of measures. Moreover, for any convex and lower semicontinuous function
	$f:\gr^{3*}\otimes\gr^3\otimes\gr^{3*}\to [0,+\ii]$, we get
	\begin{equation*}
	\int_{\cG_2(\cB)} f(A)\, dV \le \liminf_{h\to\ii}
	\int_{\cG_2(\cB)}f( A^{(h)}) \, dV^{(h)}.
	\end{equation*}
\end{theorem}
\subsection{Gradient polyconvexity}
According to references \cite{BKS,KPS,KR}, we take a continuous function
$$\hat W:\gr^{3\tim 3}\tim\gr^{3\tim 3\tim 3}\tim\gr^3\to(-\infty,+\infty], $$
and we set $\hat W=\hat W(G,\D_1,\D_2)$.
We assume also existence of four real exponents $p,q,r,s$ satisfying the inequalities
\begin{equation}\label{exp}
p>2\,,\quad q\geq{p\over p-1}\,,\quad r>1\,,\quad s>0 \end{equation}
and a positive real constant $c$ such that for every $(G,\D_1,\D_2)\in\gr^{3\tim 3}\tim\gr^{3\tim 3\tim 3}\tim\gr^3$ the following estimates holds:
$$ \hat W(G,\D_1,\D_2)\geq c\,\bigl( |G|^p+|\mathrm{cof}G|^q+(\det G)^r+(\det G)^{-s}+|\D_1|^q+|\D_2|^r\bigr)
$$
if $\det G>0$, and $\hat W(G,\D_1,\D_2)=+\infty$ if $\det G\leq 0$.
\begin{definition}\label{Dgradpoly}  With $\cB\subset\gr^3$ the domain already described, consider the functional
	$$ J(F;\mathcal{B}):=\int_\cB \hat W\bigl( F(x),\nabla[\mathrm{cof}F(x)],\nabla[\det F(x)]\bigr)\,dx $$
	defined on the class of integrable functions $F:\cB\to\gr^{3\tim 3}$ for which the approximate derivatives $\nabla[\mathrm{cof}F(x)]$, $\nabla[\det F(x)]$ exist for $\cL^3$-a.e. $x\in\cB$ and are both integrable functions in $\cB$. Then, $J(F;\mathcal{B})$ is called {\em gradient polyconvex} if the integrand $\hat W(G,\cdot,\cdot)$ is convex in
	$\gr^{3\tim 3\tim 3}\tim\gr^3$ for every $G\in \gr^{3\tim 3}$.
\end{definition}
\par To assign the Dirichlet condition, we assume that $\G_0\cup\G_1$ is an $\cH^2$-measurable partition of the $\cB$ boundary such that $\cH^2(\G_0)>0$. For some given measurable function $y_0:\G_0\to\gr$, we consider the class
$$ \begin{array}{r} \hat{\cA}_{p,q,r,s}:=\{ y\in W^{1,p}(\cB,\gr^3)  \mid   \mathrm{cof}\nabla y \in W^{1,q}(\cB,\gr^{3\tim 3})\,,\,\, \det\nabla y\in W^{1,r}(\cB)\,,
\quad\quad \\
\det\nabla y>0\,\,{\text{a.e. in }}\cB\,,\,\, (\det\nabla y)^{-1}\in L^s(\cB)\,,\,\, y=y_0 \,\,{\text{on }}\G_0 \}\,,
\end{array} $$
where $p,q,r,s$ satisfy the inequalities (\ref{exp}).

\par The following existence result has been proven in reference \cite{BKS} (see also \cite{KPS}).
\begin{theorem}\label{Tgradpoly} Under the previous assumptions, if the class $\hat{\cA}_{p,q,r,s}$ is non-empty and
	$ \inf\{ J(\nabla y;\mathcal{B})\mid y\in \hat{\cA}_{p,q,r,s}\}<\infty $,
the functional $y\mapsto J(\nabla y;\mathcal{B})$ attains a minimum in $\cA_{p,q,r,s}$.
\end{theorem}
\section{Gradient polyconvex bodies with fractures}
We now look at an energy modified by the introduction
of a varifold, through which we parametrize
possible fractured configurations with respect to the
reference one.
Specifically, we consider a curvature varifold with boundary: $V\in CV^{\ol p}_2(\cB)$.
The choice implies a fracture energy modified with respect to the Griffith one. In fact, the latter is just proportional to the crack area, which implies considering material bonds of spring-like type.
The additional presence in our case of the generalized curvature tensor implies, instead, considering beam-like material bonds for which bending effects play a role. In a certain sense, the energy we propose is a regularization of the Griffith one, since we require that the coefficient in front of the curvature tensor square modulus does not vanish.

In this setting, we look for minimizing deformations that are bounded and may admit a jump set contained in the varifold support.
We cannot assume the deformation $y$ to be a Sobolev map. More generally we require $y\in SBV(\cB,\gr^3)$. The main issue in proving existence is recovering
the weak convergence of minors. To achieve it we look at the approximate gradient and exploit Federer-Fleming's closure theorem as in Theorem~\ref{TclosG}.
On the other hand, since some properties as the bound $\Vert\mathrm{cof}\nabla y\Vert_\ii<\ii$ fails to hold, we assume $\mathrm{cof}\nabla y$ to be in the
class $GSBV$, with jump set controlled by the varifold support. In this way we recover the weak continuity of the approximate gradients $\nabla[\mathrm{cof}\nabla y_h]$
along minimizing sequences.

Our existence result below could be generalized to the case in which the crack path is described by a stratified family of varifolds in the sense introduced in references
\cite{GMMM10} and \cite{M10} (see also \cite{GMM10}).
This choice had been made we could have assigned additional curvature-type energy to the crack tip, taking possibly into account energy concentrations at tip corners, when the tip is not smooth. Also,
we could describe the formation of defects with codimension 2 in front of the crack tip, specifically dislocations nucleating in front of the tip (see \cite{GMM10}). However, for the sake of simplicity, we
restrict ourselves to the choice of a single varifold, avoiding to foresee an additional tip energy and also corner energies.
 
Consequently, we consider the energy functional
$$ \cF(y,V;\mathcal{B}):=J(\nabla y;\mathcal{B})+ \cE(V;\mathcal{B})\,, $$
where $F\mapsto J(F;\mathcal{B})$ is the functional in Definition~\ref{Dgradpoly},
and
$$ \cE(V;\mathcal{B}):=\bar{a}\m_V(\mathcal{B})+\int_{\cG_2(\cB)}a_{1}\Vert A\Vert^{\ol p}\,dV+a_{2}\Vert\pa V \Vert\,, $$
with $\bar{a}$, $a_{1}$, and $a_{2}$ positive constants.

The couples deformation-varifold are in the class $\cA_{\ol p,p,q,r,s,K,C}$ defined below.

\begin{definition} Let $\ol p>1$ and $p,q,r,s$ real exponents satisfying
	\eqref{exp}, let $K, C$ be two positive constants, and let $y_0:\G_0\to\gr$ be a given measurable function, where $\G_0\cup\G_1$ is an $\cH^2$-measurable partition of the boundary of $\cB$.
	We say that a couple $(y,V)$ belongs to the class $\cA_{\ol p,p,q,r,s,K,C}$ if the following properties hold:
	\begin{enumerate}
		\item $V=V_{\fk{C},\theta }$ is a curvature $2$-varifold with boundary in $CV^{\ol p}_2(\cB)$;
		\item $y\in\cA^1(\cB,\gr^3)$, with $\Vert y\Vert_\ii\leq K$ and $y=y_0$ on $\G_0$;
		\item $\pi_\#|\pa G_y|\leq C \cdot \mu_V$;
		\item the approximate gradient $\nabla y\in L^p(\cB,\gr^{3\tim 3})$, $\mathrm{cof}\nabla y \in L^{q}(\cB,\gr^{3\tim 3})$, and $\det\nabla y\in L^r(\cB)$;
		\item $\det\nabla y>0$ {\text{a.e. in}} $\cB$, and $(\det\nabla y)^{-1}\in L^s(\cB)$;
		\item $\mathrm{cof}\nabla y \in GSBV(\cB,\gr^{3\tim 3})$, with $|\nabla[\mathrm{cof}\nabla y]|\in L^q(\O)$;
		\item $\det\nabla y \in GSBV(\cB,\gr)$, with $\nabla[\det\nabla y]\in L^r(\O)$;
		\item $\cH^{n-1}\pri S(\mathrm{cof}\nabla y)\leq \m_V$ and $\cH^{n-1}\pri S(\det\nabla y)\leq \m_V$.
	\end{enumerate}
\end{definition}
\par Assumptions (2) and (3) imply $y\in SBV(\cB,\gr^3)$, with jump set contained in the varifold support, namely $\cH^{n-1}\pri S(y)\leq \m_V$.
Moreover, if $y\in \hat{\cA}_{p,q,r,s}$, the graph current $G_y$ has null boundary $(\pa G_y)\pri\cB\tim\gr^3=0$, see \cite[Vol.~I, Sec.~3.2.4]{GMS98}.
Therefore, taking $V=0$, i.e., in the absence of fractures, it turns out that the couple $(y,0)$ belongs to the class
$\cA_{\ol p,p,q,r,s,K,C}(\cB)$, provided that $\Vert y\Vert_\ii\leq K$, independently from the choice of $\overline p$ and $C$.

\begin{theorem}\label{Tgradpolyfrac} Under previous assumptions, if the class $\cA:=\cA_{\ol p,p,q,r,s,K,C}$ of admissible couples $(y,V)$ is non-empty and
	$ \inf\{ \cF(y,V;\mathcal{B})\mid (y,V)\in \cA\}<\infty $,
the functional $(y,V)\mapsto \cF(y,V;\mathcal{B})$ attains a minimum in $\cA$.
\end{theorem}
\begin{proof} Let $\{(y_h,V^{(h)})\}$ be a minimizing sequence in $\cA$.
	By Theorem~\ref{Tcomp1}, since $\sup_h\cE(V^{(h)};\mathcal{B})<\ii$ we can find a (not relabeled) subsequence of $\set{V^{(h)}}$ and a $2$-varifold $V=V_{\fk{C},\theta}\in CV^{\ol p}_2(\cB)$, with
	curvature $A$ and boundary $\pa V$, such that
	$V^{(h)} \rhu V$, $A^{(h)}\,dV^{(h)}\rhu A\,dV$, and $\pa V^{(h)} \rhu \pa V$ in the sense of measures, so that by lower semicontinuity
	$$ \cE(V;\mathcal{B})\leq \liminf_{h\to\ii}\cE(V^{(h)};\mathcal{B})<\ii\,. $$
	\par The domain $\cB$ being bounded, in terms of a (not relabeled) subsequence $\{y_h\}\sb \cA^1(\cB,\gr^3)$,
	we find an a.e. approximately differentiable map $y\in
	L^1(\cB,\gr^3)$ such that $y_h\to y$ {\mx{strongly in }}$L^1(\cB,\gr^3)$ and
	for any multi-indices $\a$ and $\b$, with $\vert\a\vert+\vert\b\vert=3$, functions $v^\b_{\ol\a}\in L^1(\cB)$
	such that
	$$M^\b_{\ol\a}(\nabla y_h(x))\rightharpoonup v^\b_{\ol\a}(x)\qquad{\text{weakly in }}L^1(\cB)\,. $$
	Moreover, we get the bound $ \sup_h\GM(G_{y_h})<\ii$ on the mass of the i.m. rectifiable currents $G_{y_h}$ in $\cR_3(\cB\tim\gr^3)$ carried by the $y_h$ graphs, whereas the inequalities $\pi_\#|\pa G_{y_h}|\leq C \cdot \mu_{ V^{(h)} }$ imply the bound $\sup_h\GM((\pa G_{y_h})\pri\cB\tim\gr^3)<\ii$ on the boundary current masses.
	Therefore, Theorem~\ref{TclosG} yields
	$y\in\cA^1(\cB,\gr^3)$ and $v^\b_{\ol\a}(x)=M^\b_{\ol\a}(\nabla
	y(x))$ a.e in $\cB$, for every $\a$ and $\b$, whereas $G_{y_h}\rightharpoonup G_y$ weakly in $\cD_3(\cB\tim\gr^3)$; the current $G_y$ is i.m. rectifiable in
	$\cR_3(\cB\tim\gr^3)$, and the inequality $\pi_\#|\pa G_{y}|\leq C\cdot\m_{V}$ holds true.
	\par By taking into account that $\cH^{n-1}\pri S(y_h)\leq \m_{V^{(h)}}$ and $\sup_h\Vert y_h\Vert_\ii\leq K$, the compactness theorem in $SBV$ applies to the sequence $\{y_h\}\sb SBV(\cB,\gr^3)$, yielding the convergence $Dy_h\rhu Dy$ as measures, whereas $\cH^{n-1}\pri S(y)\leq \m_{V}$ and $\Vert y\Vert_\ii\leq K$, by lower semicontinuity, and clearly  $y=y_0$ on $\G_0$.
	\par By using the uniform bound
	$$ \sup_h \int_\cB\bigl(|\nabla y_h|^p+|\mathrm{cof}\nabla y_h|^q+|\det \nabla y_h|^r\bigr)dx<\ii \,,$$
	which follows from the lower bound imposed on the density $\hat W$ of the functional $F\mapsto J(F;\mathcal{B})$, we get
	$\nabla y_h\rhu \nabla y$ in $L^p(\cB,\gr^{3\tim 3})$, $\mathrm{cof}\nabla y_h\rhu \mathrm{cof}\nabla y$ in $L^q(\cB,\gr^{3\tim 3})$, and $\det\nabla y_h\rhu \det\nabla y$ in $L^r(\cB)$.
	\par Also, the inequalities $\cH^{n-1}\pri S(\mathrm{cof}\nabla y_h)\leq \m_{V^{(h)}}$ and the lower bound on $\hat W$ imply that the sequence $\{\mathrm{cof}\nabla y_h\}\sb GSBV(\cB,\gr^{3\tim 3})$ satisfies the inequality
	$$ \sup_h \Bigl( \int_\cB \bigl(|\mathrm{cof}\nabla y_h|^q+|\nabla[\mathrm{cof}\nabla y_h]|^q\bigr)\,dx +\cH^{n-1}(S({\mathrm{cof}\nabla y_h}))\Bigr) <\ii\,. $$
	Therefore, by Theorem~\ref{TGSBV} we infer that
	\begin{itemize}
		\item $\mathrm{cof}\nabla y\in GSBV(\cB,\gr^{3\tim 3})\,,$
		
		\item $\mathrm{cof}\nabla y_h\to\mathrm{cof}\nabla y$ in $L^q(\cB,\gr^{3\tim 3})\,,$
		
		\item $\nabla[\mathrm{cof}\nabla y_h]\rhu \nabla[\mathrm{cof}\nabla y]$ weakly in $L^q(\cB,\gr^{3\tim 3\tim 3})$, and
		
		\item $\cH^{n-1}\pri S(\mathrm{cof}\nabla y)\leq \m_V$.
	\end{itemize}
	
Similarly, the inequalities $\cH^{n-1}\pri S(\det\nabla y_h)\leq \m_{V^{(h)}}$ and the lower bound on $\hat W$ imply that the sequence
	$\{\det\nabla y_h\}\sb GSBV(\cB)$ satisfies the inequality
	$$ \sup_h \Bigl( \int_\cB \bigl(|\det\nabla y_h|^r+|\nabla[\det\nabla y_h]|^r\bigr)\,dx +\cH^{n-1}(S({\det\nabla y_h})) \Bigr) <\ii\,, $$
	so that Theorem~\ref{TGSBV} entails that
	\begin{itemize}
		\item $\det \nabla y\in GSBV(\cB)\,,$
		
		\item $\det\nabla y_h\to\det\nabla y$ in $L^r(\cB)\,,$
		
		\item $\nabla[\det\nabla y_h]\rhu \nabla[\det\nabla y]$ weakly in $L^r(\cB,\gr^{3})$, and
		
		\item $\cH^{n-1}\pri S(\det\nabla y)\leq \m_V$.
	\end{itemize}

 Arguing as in the proof of Theorem~\ref{Tgradpoly}, reported in reference \cite{KPS}, we obtain $\det\nabla y>0$ a.e. in $\cB$, and
	$(\det\nabla y)^{-1}\in L^s(\cB)$, whence we get  $(y,V)\in\cA=\cA_{\ol p,p,q,r,s,K,C}$.
	
	Finally, on account of the previous convergences, the gradient polyconvexity assumption implies the lower semicontinuity inequality
	$$ J(\nabla y;\mathcal{B})\leq\liminf_{h\to\ii}J(\nabla y_h;\mathcal{B}). $$
	Then,
	$$ \cF(y,V)\leq\liminf_{h\to\ii}\cF(y_h,V^{(h)}), $$
	which is the last step in the proof.
\end{proof}
\subsection{By avoiding self-penetration}

The restriction imposed to $\det \nabla y(x)$ ensures that the deformation locally preserves orientation.
However, we have also to allow possible self-contact between distant portions of the boundary preventing at the same time self-penetration of the matter.
To this aim, in 1987 P. Ciarlet and J. Ne\v{c}as proposed the introduction of an additional constraint, namely
$$  \int_{\cB'} \det \nabla y(x)\, dx \le \cL^3(\wid y(\wid\cB'))
$$
for any sub-domain $\cB'$ of $\cB$, where $\wid \cB'$ is intersection of $\cB'$ with the domain $\wid \cB$ of Lebesgue's representative $\wid y$ of $y$ \cite{CN}.

We adopt here a weaker constraint, introduced in 1989 by M. Giaquinta, G. Modica, and J. Sou\v{c}ek \cite{GMS-Arma} (see also \cite[Vol.~II,~Sec. 2.3.2]{GMS98}). It reads
$$
\int_{\cB} f(x,u(x))\,\det \nabla y(x)\, dx \le
\int_{\gr^3}\sup_{x\in\cB} f(x,y)\,dy\,, $$
for every compactly supported smooth function $f:\cB\tim\gr^3\to[0,+\ii)$.

We thus denote by $\wid\cA_{\ol p,p,q,r,s,K,C}$ the set of couples $(y,V)\in \cA_{\ol p,p,q,r,s,K,C}$ such that the deformation map $y$ satisfies the previous inequality.

Since that constraint is preserved by the weak convergence as currents $G_{y_h}\rhu G_y$ along minimizing sequences, arguing as in Theorem~\ref{Tgradpolyfrac}
we readily obtain the following existence result.
\begin{corollary}\label{Cgradpolyfrac} Under the previous assumptions, if the class $\wid\cA:=\wid\cA_{\ol p,p,q,r,s,K,C}$ of admissible couples $(y,V)$ is non-empty and
	$ \inf\{ \cF(y,V)\mid (y,V)\in \wid\cA\}<\infty $,
	then the minimum of the functional $(y,V)\mapsto \cF(y,V)$ is attained in $\wid\cA$.
\end{corollary}

\ \ \

\textbf{Acknowledgements}. This work has been developed within the
activities of the research group in \textquotedblleft Theoretical
Mechanics\textquotedblright\ of the \textquotedblleft Centro di Ricerca
Matematica Ennio De Giorgi\textquotedblright\ of the Scuola Normale
Superiore in Pisa. PMM wishes to thank the Czech Academy of Sciences for hosting him in Prague during February 2020 as a visiting professor. We acknowledge also the support of GA\v{C}R-FWF project 19-29646L (to MK),   GNFM-INDAM (to PMM), and GNAMPA-INDAM (to DM).

\end{document}